\documentclass{amsart}

\newtheorem{theorem}{Theorem}[section]
\newtheorem{lemma}[theorem]{Lemma}
\newtheorem{proposition}[theorem]{Proposition}
\newtheorem{corollary}[theorem]{Corollary}

\theoremstyle{definition}

\newtheorem{remark}[theorem]{Remark}

\def\cocoa{{\hbox{\rm C\kern-.13em o\kern-.07em C\kern-.13em o\kern-.15em A}}}

\usepackage{comment}
\usepackage{amsmath, amssymb,url}
\usepackage{amscd,amssymb}
\usepackage{youngtab}

\begin{document}

\title[Cocharacters of representations of $\mathfrak{sl}_2(\mathbb{C})$]
{Bound for the cocharacters of the  identities of irreducible representations of $\mathfrak{sl}_2(\mathbb{C})$}

\author[M\'aty\'as Domokos]
{M\'aty\'as Domokos}
\address{Alfr\'ed R\'enyi Institute of Mathematics,
Re\'altanoda utca 13-15, 1053 Budapest, Hungary}
\email{domokos.matyas@renyi.hu}

\thanks{Partially supported by the Hungarian National Research, Development and Innovation Office,  NKFIH K 138828,  K 132002.}



\maketitle

\newenvironment{dedication}
        {\vspace{6ex}\begin{quotation}\begin{center}\begin{em}}
        {\par\end{em}\end{center}\end{quotation}}
\begin{dedication}
{Dedicated to Vesselin Drensky on his $70$th birthday}
\end{dedication}

\begin{abstract}For each irreducible finite dimensional representation of the Lie algebra $\mathfrak{sl}_2(\mathbb{C})$ of $2\times 2$ traceless matrices, an explicit uniform upper bound is given for the multiplicities in the cocharacter sequence of the polynomial identities satisfied by the given representation. 
\end{abstract}

\section{Introduction} \label{sec:intro}

Let $\rho:\mathfrak{g}\to \mathfrak{gl}(V)$ be a finite dimensional representation of the Lie algebra $\mathfrak{g}$ over a field $K$ of characteristic zero;  that is, $\mathfrak{gl}(V)=\mathrm{End}_K(V)$, the space of all $K$-linear transformations of the finite dimensional $K$-vector space $V$, viewed as a Lie algebra with Lie product $[A,B]:=A\circ B-B\circ A$ for $A,B\in \mathrm{End}_K(V)$, and $\rho$ is a 
homomorphism of Lie algebras. Denote by $F_m:=K\langle x_1,\dots,x_m\rangle$ the free associative $K$-algebra with $m$ generators. Consider $F_m$ a subalgebra of $F_{m+1}$ in the obvious way, and write 
$F:=\bigcup_{m=1}^\infty F_m$ for the free associative algebra of countable rank. 
We say that \emph{$f=0$ is an identity of the representation $\rho$ of $\mathfrak{g}$} (or briefly, of the pair $(\mathfrak{g},\rho)$) for some $f\in F_m$ if for any elements 
$A_1,\dots,A_m\in \mathfrak{g}$ we have the following equality in the associative $K$-algebra $\mathrm{End}_K(V)$: 
\[f(\rho(A_1),\dots,\rho(A_n))=0\in  \mathrm{End}_K(V).\] 
Note that an identity of the representation $\rho$ of the Lie algebra $\mathfrak{g}$ is also called in the literature a \emph{weak polynomial identity} for the pair $(\mathrm{End}_K(V),\rho(\mathfrak{g}))$.  
This notion was introduced and powerfully applied  first by Razmyslov \cite{Razmyslov:1, Razmyslov:2, Razmyslov:3, Razmyslov:4}
(see Drensky \cite{drensky} for a recent survey on weak polynomial identities).  
Set 
\[I(\mathfrak{g},\rho):=\{f\in F\mid f=0\text{ is an identity of }(\mathfrak{g},\rho)\}.\] 
Clearly $I(\mathfrak{g},\rho)$ is an ideal in $F$ stable with respect to all $K$-algebra endomorphisms of $F$ of the form $x_i\mapsto u_i$, where $u_i$ for $i=1,2,\dots$ is an element of the Lie subalgebra of $F$ generated by $x_1,x_2,\dots$. 
In particular, the general linear group $\mathrm{GL}_m(K)$ acts on $F_m$ via $K$-algebra automorphisms: for $g=(g_{ij})_{i,j=1}^m$ we have $g\cdot x_j=\sum_{i=1}^mg_{ij}x_i$, and $I(\mathfrak{g},\rho)\cap F_m$ is a $\mathrm{GL}_m(K)$-invariant subspace of $F_m$. 
The \emph{multilinear component} of $F_m$ is 
\[P_m:=\mathrm{Span}_K\{x_{\pi(1)}\cdots x_{\pi(m)}\mid \pi\in S_m\},\] 
where $S_m$ is the symmetric group of degree $m$. 
It is well known that when $\mathrm{char}(K)=0$, the ideal 
$I(\mathfrak{g},\rho)$ is determined by the multilinear components $I(\mathfrak{g},\rho)\cap P_m$, $m=1,2,\dots$. 
Identifying $S_m$ with the subgroup of permutation matrices in $\mathrm{GL}_m(K)$ we get its action on $F_m$ via $K$-algebra automorphisms (more explicitly, $\pi\in S_m$ is the automorphism of $F_m$ given by $x_i\mapsto x_{\pi(i)}$), and 
the subspaces $P_m$ and  $I(\mathfrak{g},\rho)\cap P_m$ are $S_m$-invariant. 
Define the \emph{$m$th cocharacter of $(\mathfrak{g},\rho)$} as 
\[\chi_m(\mathfrak{g},\rho):=\text{ the character of the }S_m\text{-module }P_m/(I(\mathfrak{g},\rho)\cap P_m).\] 
We call   
\[\chi(\mathfrak{g},\rho):=(\chi_m(\mathfrak{g},\rho)\mid m=1,2,\dots)\] 
the \emph{cocharacter sequence of} $(\mathfrak{g},\rho)$. 
The irreducible $S_m$-modules are labeled by partitions of $m$; let $\chi^{\lambda}$ denote the character of the irreducible 
$S_m$-module associated to the partition $\lambda=(\lambda_1,\dots,\lambda_m)\vdash m$. We have   
\[\chi_m(\mathfrak{g},\rho)=\sum_{\lambda\vdash m}\mathrm{mult}_{\lambda}(\mathfrak{g},\rho)\chi^{\lambda},\]
and we are interested in the multiplicities $\mathrm{mult}_{\lambda}(\mathfrak{g},\rho)$ of the irreducible $S_m$-characters in the 
cocharacter sequence. 
Note that the value of $\chi_m(\mathfrak{g},\rho)$ on the identity element of $S_m$ is 
\[c_m(\mathfrak{g},\rho):=\dim_K(P_m/(I(\mathfrak{g},\rho)\cap P_m)),\] 
and 
\[(c_m(\mathfrak{g},\rho)\mid m=1,2,\dots)\] 
is called the \emph{codimension sequence of} 
$(\mathfrak{g},\rho)$. It was proved by Gordienko \cite{gordienko} 
that $\lim_{m\to\infty}\root m\of{c_m(\mathfrak{g},\rho)}$ exists and is an integer.  
As is  observed in \cite[Example 3]{gordienko}, an obvious upper bound for 
$c_m(\mathfrak{g},\rho)$ can be obtained from the fact that there is a natural $K$-linear embedding 
\begin{equation}\label{eq:embedding}
P_m/(I(\mathfrak{g},\rho)\cap P_m)\hookrightarrow 
 \mathrm{Hom}_K(\rho(\mathfrak{g})^{\otimes m}, \mathrm{End}_K(V)). 
\end{equation}
Our starting observation is that the adjoint representation  of $\mathfrak{g}$ on itself  induces a natural representation of $\mathfrak{g}$ on $\rho(\mathfrak{g})^{\otimes m}$ 
(the $m$th tensor power of $\rho(\mathfrak{g})$) and on $ \mathrm{End}_K(V)$,  such that 
the image of the embedding \eqref{eq:embedding} is contained in the subspace of $\mathfrak{g}$-module homomorphisms from $\rho(\mathfrak{g})^{\otimes m}$ to $\mathrm{End}_K(V)$. So 
\eqref{eq:embedding} can be refined as 
\begin{equation}\label{eq:2embedding} 
P_m/(I(\mathfrak{g},\rho)\cap P_m)\hookrightarrow 
 \mathrm{Hom}_{\mathfrak{g}}(\rho(\mathfrak{g})^{\otimes m}, \mathrm{End}_K(V)). 
\end{equation}

This will be used to give an upper bound for the multiplicities in the cocharacter sequence 
$\chi(\frak{sl}_2(\mathbb{C}),\rho^{(d)})$ of the $d$-dimensional irreducible representation
\[\rho^{(d)}:\mathfrak{sl}_2(\mathbb{C})\to \mathfrak{gl}(\mathbb{C}^{d})=\mathbb{C}^{d\times d}\]  
of $\mathfrak{sl}_2(\mathbb{C})$  
for $d=1,2,\dots$. Note that throughout the paper we shall identify $\mathfrak{gl}(\mathbb{C}^d)$ with 
the associative algebra $\mathbb{C}^{d\times d}$ of $d\times d$ complex matrices, viewed as a Lie algebra with Lie bracket $[A,B]=AB-BA$. 

\begin{theorem}\label{thm:main} 
The multiplicity $\mathrm{mult}_{\lambda}(\frak{sl}_2(\mathbb{C}),\rho^{(d)})$ in  
$\chi(\frak{sl}_2(\mathbb{C}),\rho^{(d)})$ is non-zero only if $\lambda=(\lambda_1,\lambda_2,\lambda_3)$ (i.e. $\lambda$ has at most $3$ non-zero parts), 
and in this case we have the inequality 
\[\mathrm{mult}_{\lambda}(\frak{sl}_2(\mathbb{C}),\rho^{(d)})\le 3^{d-2}.\] 
\end{theorem} 

\begin{remark} \label{remark:1} (i) The exact values of $\mathrm{mult}_{\lambda}(\frak{sl}_2(\mathbb{C}),\rho^{(d)})$ are known for $d\le 3$. For $d=1$ all the multiplicities are obviously zero.  
It was proved in \cite{Procesi:2} (see also \cite[Exercise 12.6.12]{Drensky:2}) that 
\[\mathrm{mult}_{\lambda}(\frak{sl}_2(\mathbb{C}),\rho^{(2)})= 1 \text{ for all }\lambda=(\lambda_1,\lambda_2,\lambda_3). \]
The multiplicities $\mathrm{mult}_{\lambda}(\frak{sl}_2(\mathbb{C}),\rho^{(3)})$ are computed in 
\cite[Theorem 3.7, Proposition 3.8]{DD}. It turns out that 
$\mathrm{mult}_{\lambda}(\frak{sl}_2(\mathbb{C}),\rho^{(3)})\in \{1,2,3\}$ for each $\lambda=(\lambda_1,\lambda_2,\lambda_3)$. 

(ii) Theorem~\ref{thm:main} shows in particular that for each dimension $d$, there is a uniform bound (depending on $d$ only)  for the multiplicities $\mathrm{mult}_{\lambda}(\frak{sl}_2(\mathbb{C}),\rho^{(d)})$. For comparison  we mention that the 
multiplicities in the cocharacter sequence of the ordinary polynomial identities of $2\times 2$ matrices are unbounded: 
see \cite{Drensky:84} and \cite{Formanek}. For example, for any partition $\lambda=(\lambda_1,\lambda_2)$ with $\lambda_2>0$, the multiplicity is $(\lambda_1-\lambda_2+1)\lambda_2$. On the other hand, the cocharacter multiplicities of any PI algebra are polynomially bounded by \cite{berele-regev}. 

(iii) There is no uniform upper bound independent of $d$ for the multiplicities $\mathrm{mult}_{\lambda}(\frak{sl}_2(\mathbb{C}),\rho^{(d)})$, because by  Proposition~\ref{prop:lowerbound}, $\max\{\mathrm{mult}_{\lambda}(\frak{sl}_2(\mathbb{C}),\rho^{(d)})\mid m=1,2,\dots,\ \lambda\vdash m\}\ge d-1$ for $d\ge 2$. 

(iv) The irreducible representations of $\mathfrak{sl}_2(\mathbb{C})$ are defined over $\mathbb{Q}$. 
For any field $K$ of characteristic zero and any positive integer $d$, the Lie algebra $\mathfrak{sl}_2(K)$ has a unique 
(up to isomorphism) $d$-dimensional irreducible representation $\rho^{(d)}_K$ over $K$. 
By well-known general arguments, the multiplicities $\mathrm{mult}_{\lambda}(\frak{sl}_2(K),\rho^{(d)}_K)$ do not depend on $K$. 
Therefore Theorem~\ref{thm:main} implies that 
$\mathrm{mult}_{\lambda}(\frak{sl}_2(K),\rho^{(d)}_K)\le 3^{d-2}$ for any field $K$ of characteristic zero. 

(v) A different interpretation and approach to the study of 
 $\mathrm{Hom}_{\mathfrak{g}}(\rho(\mathfrak{g})^{\otimes m}, \mathrm{End}_K(V))$ 
 for $\mathfrak{g}=\mathfrak{sl}_2(\mathbb{C})$ and $\rho=\rho^{(d)}$ is given in our parallel 
 preprint \cite{Domokos}, using classical invariant theory. 
\end{remark} 

We close the introduction by mentioning the recent paper of da Silva Macedo and Koshlukov \cite[Theorem 3.7]{dSM-K}, where the codimension growth of polynomial identities of representations of Lie algebras is studied. 
In particular, in \cite[Theorem 3.7]{dSM-K} the identities of representations of $\frak{sl}_2(\mathbb{C})$ play a decisive role.

\section{Matrix  computations} 

Denote by $\widetilde\rho^{(d)}:\mathfrak{sl}_2(\mathbb{C})\to 
\mathfrak{gl}(\mathbb{C}^{d\times d})$ 
the representation given by 
\begin{equation}\label{eq:rhotilde}  
\widetilde\rho^{(d)}(A)(L)= \rho^{(d)}(A)    L-L   \rho^{(d)}(A) \text{ for }A\in \mathfrak{sl}_2(\mathbb{C}), \ 
L\in \mathbb{C}^{d\times d}. 
\end{equation}  
We have $\widetilde\rho^{(d)}\cong \rho^{(d)}\otimes {\rho^{(d)}}^*$. 
The representations of $\mathfrak{sl}_2(\mathbb{C})$ are self-dual, and so by the Clebsch-Gordan rules we have 
\begin{equation}\label{eq:clebsch-gordan}
\widetilde\rho^{(d)}\cong \rho^{(d)}\otimes {\rho^{(d)}}\cong\bigoplus_{n=1}^d\rho^{(2n-1)}.
\end{equation} 
We shall need an explicit decomposition of $\mathbb{C}^{d\times d}$ as a direct sum of minimal 
$\widetilde\rho^{(d)}$-invariant subspaces. 

Set
\[e:=\left(\begin{array}{cc}0 & 1 \\0 & 0\end{array}\right), \quad f:=\left(\begin{array}{cc}0 & 0 \\1 & 0\end{array}\right), 
\quad h:=\left(\begin{array}{cc}1 & 0 \\0 & -1\end{array}\right),\]
so $e,f,h$ is a $\mathbb{C}$-vector space basis of $\mathfrak{sl}_2(\mathbb{C})$, with $[h,e]=2e$, $[h,f]=-2f$, and $[e,f]=h$.  

Recall that given a representation $\psi:\mathfrak{sl}_2(\mathbb{C})\to \mathfrak{gl}(V)$, by a \emph{highest weight vector} we mean a non-zero element $w\in V$ such that $\psi(e)(w)=0\in V$ and $\psi(h)(w)=n w$ for some 
non-negative integer $n$ (the non-negative integer $\lambda$ is called the \emph{weight} of $w$); in this case 
$w$ generates a minimal $\mathfrak{sl}_2(\mathbb{C})$-invariant subspace in $V$, on which the representation of $\mathfrak{sl}_2(\mathbb{C})$ is isomorphic to $\rho^{(n+1)}$. Moreover, any finite dimensional irreducible $\mathfrak{sl}_2(\mathbb{C})$-module contains a unique (up to non-zero scalar multiples) highest weight vector.

\begin{lemma}\label{lemma:sl2decomposition of end} 
Consider the $\mathfrak{sl}_2(\mathbb{C})$-module  $\mathbb{C}^{d\times d}$ via the representation $\widetilde\rho^{(d)}$. To simplify notation set $\rho:=\rho^{(d)}$ and 
$\widetilde\rho:=\widetilde\rho^{(d)}$. 
\begin{itemize}
\item[(i)] $\rho(e)^n$ is a highest weight vector in $\mathbb{C}^{d\times d}$ of weight $2n$ 
for $n=0,1,\dots,d-1$.  
\item[(ii)] $\rho(e)^{n-1}$ generates a minimal $\widetilde\rho$-invariant subspace $V_n$ on which $\mathfrak{sl}_2(\mathbb{C})$ acts via $\rho^{(2n-1)}$ for $n=1,\dots,d$. 
\item[(iii)] $\mathbb{C}^{d\times d}=\bigoplus_{n=1}^dV_n$. 
\item[(iv)] For $L_1\in V_{n_1}$ and $L_2\in V_{n_2}$ with $1\le n_1\neq n_2\le d$ we have 
$\mathrm{Tr}(L_1  L_2)=0$. 
\end{itemize}
\end{lemma}

\begin{proof} (i) We have $\widetilde\rho(e)(\rho(e)^n)=\rho(e) \rho(e)^n-\rho(e)^n \rho(e)=0$ 
and 
\[\widetilde\rho(h)(\rho(e)^n)=\rho([h,e])  \rho(e)^{n-1}+\rho(e) \rho([h,e]) \rho(e)^{n-2}+\cdots+
\rho(e)^{n-1} \rho([h,e])=2n\rho(e)^n.\] 
This shows that $\rho(e)^n$ is a highest weight vector of weight $2n$ for the representation $\widetilde\rho$. 

(ii) Statement (i) implies that $\rho(e)^{n-1}$ generates an irreducible $\mathfrak{sl}_2(\mathbb{C})$-submodule of $\widetilde\rho$ isomorphic to $\rho^{(2n-1)}$ for $n=1,\dots,d$. 

(iii) follows from (ii) and \eqref{eq:clebsch-gordan}. 

(iv)  Consider the symmetric non-degenerate bilinear form 
\[\beta:\mathbb{C}^{d\times d}\times \mathbb{C}^{d\times d}\to \mathbb{C}, \quad  
(L,M)\mapsto \mathrm{Tr}(LM).\]  
Note that $\beta$ is $\widetilde\rho$-invariant: 
\begin{align*}\beta([\rho(A),L],M)+\beta(L,[\rho(A),M])=\mathrm{Tr}([\rho(A),L]M)+
\mathrm{Tr}(L[\rho(A),M])\\ 
=\mathrm{Tr}([\rho(A),LM])=0 \quad \text{ for any }A\in \mathfrak{sl}_2(\mathbb{C}). 
\end{align*}  
The radical of the bilinear form  $\beta_{V_n}:V_n\times V_n\mapsto \mathbb{C}$ (the restriction of $\beta$ to $V_n\times V_n$) is a $\widetilde\rho$-invariant subspace in $V_n$, so it is either $V_n$ or $\{0\}$. We claim that it is not $V_n$. Indeed, $V_n$ contains a non-zero diagonal matrix $D$ with real entries, since the zero weight subspace in $\mathbb{C}^{d\times d}$ (with respect to $\widetilde\rho(h)$)  is the subspace of diagonal matrices, and $V_n$ intersects the zero-weight space in a $1$-dimensional subspace (defined over the reals). Now being a sum of squares of non-zero real numbers, $0\neq \mathrm{Tr}(D^2)=\beta(D,D)$. 
Thus $\beta_{V_n}$ is non-degenerate. 
The representation $\widetilde\rho$ is multiplicity free by \eqref{eq:clebsch-gordan}, and by (ii) and (iii), 
every $\widetilde\rho$-invariant subspace is of the form $\sum_{j\in J}V_j$ for some subset $J\subseteq \{1,2,\dots,d\}$. 
As we showed above, the orthogonal complement of $V_n$ (with respect to $\beta$) is disjoint from $V_n$, so it is the sum of the other minimal invariant subspaces $V_j$, $j\in \{1,\dots,d\}\setminus\{n\}$. 
\end{proof} 

The representation $\rho^{(2)}$ is the defining representation of $\mathfrak{sl}_2(\mathbb{C})$ on $\mathbb{C}^2$, 
and $\rho^{(d)}$ is the $(d-1)$th symmetric tensor power of $\rho^{(2)}$. Denote by $x,y$ the standard basis vectors in $\mathbb{C}^2$, and take the basis $x^{d-1},x^{d-1}y,\dots,y^{d-1}$ in the $(d-1)$th symmetric tensor power of $\mathbb{C}^{2}$. Then 
denoting by 
$E_{i,j}$  the matrix unit with entry $1$ in the $(i,j)$ position and zeros in all other positions, the  representation 
$\rho^{(d)}$ as a matrix representation $\mathfrak{sl}_2(\mathbb{C})\to \mathbb{C}^{d\times d}$ is given as follows:   
\[\rho^{(d)}(e)=\sum_{i=1}^{d-1}iE_{i,i+1},\quad \rho^{(d)}(f)=\sum_{i=1}^{d-1}(d-i)E_{i+1,i}, \quad 
\rho^{(d)}(h)=\sum_{i=1}^d(d+1-2i)E_{i,i}\] 

\begin{lemma}  \label{lemma:products of length d}
For $d\ge 3$ the $\mathbb{C}$-vector space $\mathbb{C}^{d\times d}$ is spanned by 
\[\{\rho^{(d)}(A_1) \cdots \rho^{(d)}(A_{d-1})\mid  A_1,\dots,A_{d-1}\in \mathfrak{sl}_2(\mathbb{C})\}.\] 
\end{lemma}

\begin{proof} To simplify the notation write $\rho:=\rho^{(d)}$ and $\widetilde\rho:=\widetilde\rho^{(d)}$. 
Let $\mathcal{L}$ be the subspace of $\mathbb{C}^{d\times d}$ spanned by 
the products $\rho(A_1) \cdots\rho(A_{d-1})$, where $A_1,\dots,A_{d-1}\in \mathfrak{sl}_2(\mathbb{C})$. 
Clearly $\mathcal{L}$ is a $\widetilde\rho$-invariant subspace of $\mathbb{C}^{d\times d}$. 
Since the representation $\widetilde\rho$ is multiplicity free by \eqref{eq:clebsch-gordan}, we have 
$\mathcal{L}=\sum_{j\in J}V_j$ for some subset $J\subseteq \{1,2,\dots,d\}$ 
by Lemma~\ref{lemma:sl2decomposition of end} (ii) and (iii). 
Therefore to prove the equality 
$\mathcal{L}=\mathbb{C}^{d\times d}$ it is sufficient to show that $\mathcal{L}\cap V_n\neq\{0\}$ for each $n=1,\dots,d$, 
or equivalently, that $\mathcal{L}$ is not contained in $\sum_{j\in \{1,\dots,d\}\setminus \{n\}}V_j$. 
Since $V_d$ is generated by $\rho(e)^{d-1}\in \mathcal{L}$, we have $V_d\subseteq \mathcal{L}$. 
Moreover, to prove 
$\mathcal{L}\nsubseteq \sum_{j\in \{1,\dots,d\}\setminus \{n+1\}}V_j$ for $n\in\{0,1,\dots,d-2\}$, 
it is sufficient to present an element $L_n\in \mathcal{L}$  with 
$\mathrm{Tr}(\rho(e)^n L_n)\neq 0$ by Lemma~\ref{lemma:sl2decomposition of end} (i),  (ii) and (iv). 
We shall give below such elements $L_n\in \mathcal{L}$ for 
$n=0,1,\dots,d-2$. 

For $n=1,\dots,d-1$ we have 
\[\rho(e)^n=\sum_{j=1}^{d-n}j\cdot (j+1)\cdots (j+n-1)E_{j,j+n}\] 
\[\rho(f)^n=\sum_{j=1}^{d-n} (d-j) \cdot (d-j-1)\cdots (d-j-n+1)E_{j+n.j}\]
and $\rho(e)^0=I_d=\rho(f)^0$, where $I_d$ is the $d\times d$ identity matrix.  
It follows that for $n=1,\dots,d-1$, 
\[\rho(e)^n \rho(f)^n=\sum_{j=1}^{d-n}j(j+1)\cdots (j+n-1)\cdot (d-j)(d-j-1)\cdots(d-j-n+1)E_{j,j}\] 
is a diagonal matrix with non-negative integer entries, and the $(1,1)$-entry is positive. The same holds for 
$\rho(e)^0\rho(f)^0=I_d$. 
For $n$ with $d-1-n$ even, $\rho(h)^{d-1-n}$ is the square of a diagonal 
matrix with integer entries, and its $(1,1)$-entry is positive. Hence 
$\mathrm{Tr}(\rho(e)^n \rho(f)^n\rho(h)^{d-1-n})\neq 0$, being a positive integer. So in this case we may take $L_n:=\rho(f)^n\rho(h)^{d-1-n}$. 
For $n<d-2$ with $d-1-n$ odd, note that $\rho(e) \rho(f)-\rho(f)\rho(e)=\rho([e,f])=\rho(h)$, 
and thus 
\[\rho(f)^n \rho(h)^{d-2-n}=\rho(f)^n\rho(h)^{d-3-n}(\rho(e)\rho(f)-\rho(f)\rho(e))\] 
also belongs to $\mathcal{L}$. Since  $\rho(h)^{d-2-n}$ is a  diagonal matrix with non-negative integer entries, and with a positive $(1,1)$-entry,
we may take $L_n:=\rho(f)^n \rho(h)^{d-2-n}$ in this case. 
It remains to deal with the case $n=d-2$. 
Then 
\[\rho(e)^{d-2}\rho(f)^{d-2}=(d-1)((d-2)!)^2\cdot (E_{1,1}+E_{2,2}),\] 
hence taking $L_{d-2}:=\rho(f)^{d-2} \rho(h)$ we get 
\begin{align*}\mathrm{Tr}(\rho(e)^{d-2} L_{d-2})&=\mathrm{Tr}((d-1)((d-2)!)^2\cdot((d-1)E_{11}+(d-3)E_{22})
\\ &=(2d-4)(d-1)((d-2)!)^2,
\end{align*} 
which is non-zero for $d\ge 3$. 
This finishes the proof of the equality $\mathcal{L}=\mathbb{C}^{d\times d}$.  
\end{proof}

\section{Adjoint invariants}

Denote by $\mathrm{ad}:\mathfrak{sl}_2(\mathbb{C})\to \mathfrak{gl}(\mathfrak{sl}_2(\mathbb{C}))$ the \emph{adjoint representation} of $\mathfrak{sl}_2(\mathbb{C})$ on itself, so $\mathrm{ad}(A)(B)=[A,B]$ for $A,B\in \mathfrak{sl}_2(\mathbb{C})$. Take the $n$-fold direct sum 
$\mathrm{ad}^{\oplus n}:\mathfrak{sl}_2(\mathbb{C})\to \mathfrak{gl}(\mathfrak{sl}_2(\mathbb{C})^{\oplus n})$ of the adjoint representation, and write 
$\mathcal{O}[\mathfrak{sl}_2(\mathbb{C})^n]^{\mathfrak{sl}_2(\mathbb{C})}$ for the algebra of $\mathrm{ad}^{\oplus n}$-invariant polynomial functions on $\mathfrak{sl}_2(\mathbb{C})^{\oplus n}$. 
There is a right action of 
$\mathrm{GL}_n(\mathbb{C})$  on $\mathfrak{sl}_2(\mathbb{C})^n$ that commutes with $\mathrm{ad}^{\oplus n}$: for $g=(g_{ij})_{i,1=1}^n$ and $(A_1,\dots,A_n)\in 
\mathfrak{sl}_2(\mathbb{C})^n$ we have 
\[(A_1,\dots,A_n)\cdot g:=(\sum_{i=1}^ng_{i1}A_i,\dots,\sum_{i=1}^ng_{in}A_i).\] 
This induces a left $\mathrm{GL}_n(\mathbb{C})$-action on the coordinate ring $\mathcal{O}[\mathfrak{sl}_2(\mathbb{C})^n]$: for $g\in \mathrm{GL}_n(\mathbb{C})$, 
$f\in \mathcal{O}[\mathfrak{sl}_2(\mathbb{C})^n]$ and $(A_1,\dots,A_n)\in  \mathfrak{sl}_2(\mathbb{C})^n$ we have 
$(g\cdot f)(A_1,\dots,A_n)=f((A_1,\dots,A_n)\cdot g)$. 

\begin{lemma} \label{lemma:iota}
Consider the linear map 
$\iota:F_m=\mathbb{C}\langle x_1,\dots,x_m\rangle\to  \mathcal{O}[\mathfrak{sl}_2(\mathbb{C})^{m+d-1}]$  
given by 
\[\iota(f)(A_1,\dots,A_{m+d-1})=\mathrm{Tr}(f(\rho^{(d)}(A_1),\dots,\rho^{(d)}(A_m))\cdot  \rho^{(d)}(A_{m+1}) \cdots \rho^{(d)}(A_{m+d-1}))\] 
for  $f\in F_m$ and 
$(A_1,\dots,A_{m+d-1})\in \mathfrak{sl}_2(\mathbb{C})^{m+d-1}$. It has the following properties:  
\begin{itemize} 
\item[(i)] The image of $\iota$ is contained in the subalgebra 
$\mathcal{O}[\mathfrak{sl}_2(\mathbb{C})^{m+d-1}]^{\mathfrak{sl}_2(\mathbb{C})}$ of $\mathfrak{sl}_2(\mathbb{C})$-invariants.  
\item[(ii)] For $d\ge 3$ the kernel of $\iota$ is the ideal $I(\mathfrak{sl}_2(\mathbb{C}),\rho^{(d)})\cap F_m$. 
\item[(iii)] The map $\iota$ is $\mathrm{GL}_m(\mathbb{C})$-equivariant, where we restrict the $\mathrm{GL}_{m+d-1}(\mathbb{C})$-action on $\mathcal{O}[\mathfrak{sl}_2(\mathbb{C})^{m+d-1}]$ to the subgroup 
$\mathrm{GL}_m(\mathbb{C})\cong\{\left(\begin{array}{cc}g & 0 \\0 & I_{d-1}\end{array}\right)\mid g\in \mathrm{GL}_m(\mathbb{C})\}$ in 
$\mathrm{GL}_{m+d-1}(\mathbb{C})$.  
\end{itemize}
\end{lemma} 

\begin{proof}  
For notational simplicity we shall write $\rho$ instead of $\rho^{(d)}$.

(i) By linearity of $\iota$ it is sufficient to show that $\iota(x_{i_1}\cdots x_{i_k})$ is an $\mathfrak{sl}_2(\mathbb{C})$-invariant for any $i_1,\dots,i_k\in\{1,\dots,m\}$.   Setting 
$n=k+d-1$, $B_1=A_{i_1},\dots,B_k=A_{i_k},B_{k+1}=A_{m+1},\dots,B_n=A_{m+d-1}$ we have 
\begin{equation}\label{eq:xi1...xik}
\iota(x_{i_1}\cdots x_{i_k})(A_1,\dots,A_{m+d})=\mathrm{Tr}(\rho(B_1)\cdots\rho(B_n)).
\end{equation} 
For any $X\in \mathfrak{sl}_2(\mathbb{C})$ we have 
\begin{align*}
0&=\mathrm{Tr}([\rho(X),\rho(B_1)\cdots\rho(B_n)])
\\ &=\mathrm{Tr}(\sum_{j=1}^n\rho(B_1)\cdots\rho(B_{j-1})[\rho(X),\rho(B_j)]\rho(B_{j+1})\cdots\rho(B_n))
\\  &=\mathrm{Tr}(\sum_{j=1}^n\rho(B_1)\cdots\rho(B_{j-1})\rho([X,B_j])\rho(B_{j+1})\cdots\rho(B_n))
\\ &=\sum_{j=1}^n\mathrm{Tr}(\rho(B_1)\cdots\rho(B_{j-1})\rho([X,B_j])\rho(B_{j+1})\cdots\rho(B_n)).
\end{align*} 
The equalities \eqref{eq:xi1...xik} and 
\[\sum_{j=1}^n\mathrm{Tr}(\rho(B_1)\cdots\rho(B_{j-1})\rho([X,B_j])\rho(B_{j+1})\cdots\rho(B_n))=0\] 
mean that $\iota(x_{i_1}\cdots x_{i_k})$ is $\mathfrak{sl}_2(\mathbb{C})$-invariant, so (i) holds. 

(ii) Suppose that $f\in \ker(\iota)$. Then $\mathrm{Tr}(f(\rho(A_1),\dots,\rho(A_m))B)=0$ for 
all $B\in \mathbb{C}^{d\times d}$ by Lemma~\ref{lemma:products of length d}. By non-degeneracy of the trace we get 
$f(\rho(A_1),\dots,\rho(A_m))=0$. That is, $f\in I(\mathfrak{sl}_2(\mathbb{C}),\rho)$. Thus 
$\ker(\iota)\subseteq I(\mathfrak{sl}_2(\mathbb{C}),\rho)\cap F_m$. The reverse inclusion 
$ I(\mathfrak{sl}_2(\mathbb{C}),\rho)\cap F_m\subseteq \ker(\iota)$ is obvious. 

(iii) Take $g=(g_{ij})_{i,j=1}^m\in \mathrm{GL}_m(\mathbb{C})$. 
For $f\in F_m$ and $(A_1,\dots,A_m)\in \mathfrak{sl}_2(\mathbb{C})^m$ we have 
(by linearity of $\rho$) 
\begin{align*} 
&\iota(g\cdot f)(A_1,\dots,A_m)\\ 
&=\mathrm{Tr}(f(\sum_{i=1}^mg_{i1}\rho(A_i),\dots,\sum_{i=1}^mg_{im}\rho(A_i))\cdot  \rho(A_{m+1}) \cdots \rho(A_{m+d}))
\\ &=\mathrm{Tr}(f( \rho(\sum_{i=1}^mg_{i1}(A_i)),\dots, \rho(\sum_{i=1}^mg_{im}(A_i)) \cdot \rho(A_{m+1})\cdots \rho(A_{m+d}))
\\ &=(g\cdot \iota(f))(A_1,\dots,A_m).
\end{align*}
This shows (iii). 
\end{proof} 

Restricting the action of $\mathrm{GL}_n(\mathbb{C})$ on $\mathcal{O}[\mathfrak{sl}_2(\mathbb{C})^n]$ to the subgroup of diagonal matrices we get an $\mathbb{N}_0^n$-grading on $\mathcal{O}[\mathfrak{sl}_2(\mathbb{C})^n]$, preserved by the action of $\mathfrak{sl}_2(\mathbb{C})$. Denote by $\mathcal{O}[\mathfrak{sl}_2(\mathbb{C})^n]_{(1^n)}$ the multihomogeneous component of multidegree $(1,\dots,1)$; this is the space of $n$-linear functions on $\mathfrak{sl}_2(\mathbb{C})$. The spaces $\mathcal{O}[\mathfrak{sl}_2(\mathbb{C})^n]_{(1^n)}$ and $\mathcal{O}[\mathfrak{sl}_2(\mathbb{C})^n]_{(1^n)}^{\mathfrak{sl}_2(\mathbb{C})}$ are $S_n$-invariant (where we restrict the $\mathrm{GL}_n(\mathbb{C})$-action to its subgroup $S_n$ of permutation matrices).   
Lemma~\ref{lemma:iota} has the following immediate consequence: 

\begin{corollary}\label{cor:res-iota} 
For $d\ge 3$ the restriction of $\iota$ to the multilinear component $P_m$ of $\mathbb{C}\langle x_1,\dots,x_m\rangle$ 
factors through an $S_m$-equivariant $\mathbb{C}$-linear embedding 
\[\bar\iota:P_m/(I(\mathfrak{sl}_2(\mathbb{C}))\cap P_m)\to 
\mathcal{O}[\mathfrak{sl}_2(\mathbb{C})^{m+d-1}]_{(1^{m+d-1})}^{\mathfrak{sl}_2(\mathbb{C})}\]
where on the right hand side we consider the restriction of the $S_{m+d-1}$-action to its subgroup $S_m$ 
(the stabilizer in $S_{m+d-1}$ of the elements $m+1,m+2,\dots,m+d-1$). 
\end{corollary} 

For a partition $\lambda\vdash m$ denote by $r(\lambda)$ the multiplicity of $\chi^{\lambda}$ in the restriction to $S_m$ of the $S_{m+d-1}$-module $\mathcal{O}[\mathfrak{sl}_2(\mathbb{C})^{m+d-1}]_{(1^{m+d-1})}^{\mathfrak{sl}_2(\mathbb{C})}$. 
Corollary~\ref{cor:res-iota} immediately implies the following: 
\begin{corollary}\label{cor:m(lambda)<r(lambda)}  For $d\ge 3$ and any partition $
\lambda\vdash m$ we have the inequality 
\[\mathrm{mult}_{\lambda}(\mathfrak{sl}_2(\mathbb{C}),\rho^{(d)})\le r(\lambda).\]
\end{corollary} 

The $S_n$-character of 
$\mathcal{O}[\mathfrak{sl}_2(\mathbb{C})^n]_{(1^n)}^{\mathfrak{sl}_2(\mathbb{C})}$ is known: 

\begin{proposition}\label{prop:Sn-character} 
For a partition $\lambda\vdash n$ denote by $\nu(\lambda)$ the multiplicity of $\chi^{\lambda}$ in the $S_n$-character of 
$\mathcal{O}[\mathfrak{sl}_2(\mathbb{C})^n]_{(1^n)}^{\mathfrak{sl}_2(\mathbb{C})}$. Then we have 
\[\nu(\lambda)=\begin{cases}1\text{ for }\lambda=(\lambda_1,\lambda_2,\lambda_3) \text{ with } \lambda_1\equiv\lambda_2\equiv\lambda_3 \text{ modulo } 2 
\\ 0 \text{ otherwise.} \end{cases} \] 
\end{proposition} 

\begin{proof} 
The $\mathrm{GL}_n(\mathbb{C})$-module structure of $\mathcal{O}[\mathfrak{sl}_2(\mathbb{C})^n]^{\mathfrak{sl}_2(\mathbb{C})}$ is given for example in \cite[Theorem 2.2]{Procesi:2}. The isomorphism types of the irreducible $\mathrm{GL}_n(\mathbb{C})$-module direct summands of the degree $n$ homogeneous component of $\mathcal{O}[\mathfrak{sl}_2(\mathbb{C})^n]$ are labeled by partitions of $n$ with at most $3$ non-zero parts. The multiplicity $\mu(\lambda)$ of the irreducible $\mathrm{GL}_n(\mathbb{C})$-module $W_{\lambda}$ in the degree $n$ homogeneous component of $\mathcal{O}[\mathfrak{sl}_2(\mathbb{C})^n]^{\mathfrak{sl}_2(\mathbb{C})}$ is $1$ if $\lambda_1,\lambda_2,\lambda_3$ have the same parity and is zero otherwise.  Note finally that the multilinear component of $W_{\lambda}$ is $S_n$-stable, and its $S_n$-character is $\chi^{\lambda}$ 
(see for example \cite[Corollary 6.3.11]{AGPR}). 
\end{proof}

Following \cite[Section I.1]{Macdonald} for partitions $\lambda \vdash n$ and $\mu\vdash k$ we write 
$\lambda\subset \mu$ is $\lambda_i\le\mu_i$ for all $i$. Moreover,  given $\lambda\vdash m$ and $\mu\vdash m+d-1$ 
with $\lambda\subset \mu$, by a \emph{standard tableau of shape} $\mu/\lambda$ we mean a sequence 
$\lambda^{(0)}\subset\lambda^{(1)}\subset\dots\subset\lambda^{(d-1)}$ of partitions $\lambda^{(i)}\vdash m+i$, 
where $\lambda^{(0)}=\lambda$, $\lambda^{(d-1)}=\mu$. 
By the well-known branching rules for the symmetric group, for $\lambda\vdash m$ the multiplicity of 
$\chi^{\lambda}$ in the restriction to $S_m$  of the irreducible $S_{m+d-1}$-character $\chi^{\mu}$ equals the number of standard tableaux of shape $\mu/\lambda$ (see for example \cite[Theorem 6.4.11]{AGPR}). 
Therefore Proposition~\ref{prop:Sn-character} has the following consequence. 

\begin{corollary} \label{cor:Sm-character}
We have the equality 
\begin{align*} 
r(\lambda)=|\{ T\mid & T \text{ is a standard skew tableau of shape }\mu/\lambda,\\  
&\mu\vdash m+d-1,  \  \mu=(\mu_1,\mu_2,\mu_3),\ \mu_1\equiv\mu_2\equiv\mu_3\textrm{ modulo }2\}|.\end{align*}
\end{corollary}

\begin{corollary}\label{cor:3^d} 
For $d\ge 3$ we have the inequality $r(\lambda)\le 3^{d-2}$. 
\end{corollary} 

\begin{proof} 
Associate to a standard skew tableau $T=\lambda^{(0)}\subset\lambda^{(1)}\subset\dots\subset\lambda^{(d-1)}$ of shape $\mu/\lambda$, where $\mu=(\mu_1,\mu_2,\mu_3)\vdash m+d-1$ and $\mu_1\equiv\mu_2\equiv\mu_3\textrm{ modulo }2$ 
the function $f_T:\{1,\dots,d-1\}\to \{1,2,3\}$, which maps  $j\in \{1,\dots,d-1\}$ to the unique $i\in\{1,2,3\}$ such that the $i$th component of  the partition $\lambda^{(j)}$ is $1$ greater than the $i$th component of $\lambda^{(j-1)}$. 
The assignment $T\mapsto f_T$ is obviously an injective map from the set of standard skew tableaux of shape $\mu/\lambda$ into the set of functions $\{1,\dots,d-1\}\to \{1,2,3\}$. We claim that at most $3^{d-2}$ functions are contained in the image of this map. Indeed, if the three parts of $\lambda^{(d-3)}$ have the same parity, then 
$(f_T(d-2),f_T(d-1))\in \{(1,1),(2,2),(3,3)\}$, since the three parts of $\mu=\lambda^{(d-1)}$ must have the same parity. If the three parts of $\lambda^{(d-3)}$ do not have the same parity, say 
the first two components of $\lambda^{(d-3)}$ have the same parity, and the third part has the opposite parity, then 
 $(f_T(d-2),f_T(d-1))\in \{(1,2),(2,1)\}$. 
Hence $r(\lambda)$ is not greater than $3$-times the number of functions from a $(d-3)$-element set to a $3$-element set. 
Thus  $r(\lambda)\le 3^{d-2}$. 
\end{proof} 

\subsection{Proof of Theorem~\ref{thm:main}} 
For $d\ge 3$ the statement follows from Corollary~\ref{cor:m(lambda)<r(lambda)} and Corollary~\ref{cor:3^d}. 
For the cases $d\le 3$ see Remark~\ref{remark:1} (i). 

\section{A lower bound} 

\begin{proposition}\label{prop:lowerbound} 
For $d\ge 2$ we have the equality 
\[\mathrm{mult}_{(d-1,1)}(\frak{sl}_2(\mathbb{C}),\rho^{(d)})=d-1.\] 
\end{proposition} 

\begin{proof} 
For $k=0,1,\dots,d-2$ consider the element 
\[w_k:=x_1^k[x_1,x_2]x_1^{d-2-k}\in \mathbb{C}\langle x_1,x_2 \rangle=F_2.\] 
These elements are $\mathrm{GL}_2(\mathbb{C})$-highest weight vectors with weight $(d-1,1)$, hence each generates an irreducible $\mathrm{GL}_2(\mathbb{C})$-submodule isomorphic to $W_{(d-1,1)}$ (see the proof of Proposition~\ref{prop:Sn-character} for the notation $W_{\lambda}$: it is the polynomial $\mathrm{GL}_2(\mathbb{C})$-module with highest weight $\lambda=(\lambda_1,\lambda_2)$). Moreover, they are linearly independent modulo the ideal  
$I(\mathfrak{sl}_2(\mathbb{C}),\rho^{(d)})$: indeed, make the substitution $x_1\mapsto \rho(h)$, $x_2\mapsto \rho(e)$. 
Then we get  
\begin{align*}w_k(\rho(h),\rho(e))&= 
(\sum_{i=1}^d(d+1-2i)E_{i,i})^k\cdot (2\sum_{i=1}^{d-1}iE_{i,i+1})\cdot (\sum_{i=1}^d(d+1-2i)E_{i,i})^{d-2-k}
\\ &=2\sum_{i=1}^{d-1}i(d+1-2i)^k(d-1-2i)^{d-2-k}E_{i,i+1}. 
\end{align*} 
Denote by $Z=(Z_{i,j})_{i,j=1}^{d-1}$ the $(d-1)\times (d-1)$ matrix whose $(i,k+1)$ entry is the $(i,i+1)$-entry of 
$ w_k(\rho(h),\rho(e))$ (i.e. the coefficient of $E_{i.i+1}$ on the right hand side of the above equality). 
If $i\neq \frac{d-1}{2}$, then   
\[Z_{i,k+1}=2(d-1-2i)^{d-2} \cdot \left(\frac{d+1-2i}{d-1-2i}\right)^k.\]
Thus when $d$ is even, $Z$ is obtained from a Vandermonde matrix via multiplying each row by a non-zero integer. 
Since the numbers $\frac{d+1-2i}{d-1-2i}$, $i=1,\dots,d-1$  are distinct, we conclude that $\det(Z)\neq 0$. 
When $d=2f-1$ is odd, the $(f-1)$th row of $Z$ is  
\[(0,\dots,0,2(f-1)2^{d-2}).\] 
Expand the determinant of $Z$ along this row; the $(d-2)\times (d-2)$ minor of $Z$ obtained by removing the 
$(f-1)th$ row and the last column of $Z$ is again obtained from a Vandermonde matrix by multiplying each row by a non-zero integer. So $\det(Z)$ is non-zero also when $d$ is odd. 
This shows that the elements  $w_k(\rho(h),\rho(e))$, $k=0,1,\dots,d-2$ are linearly independent in $\mathbb{C}^{d\times d}$. Consequently, no non-trivial linear combination of   $w_0,w_1,\dots,w_{d-2}$ belongs to 
$I(\mathfrak{sl}_2(\mathbb{C}),\rho^{(d)})$. 
It follows that $F_2/(I(\mathfrak{sl}_2(\mathbb{C}),\rho^{(d)})\cap F_2)$ contains the 
irreducible $\mathrm{GL}_2(\mathbb{C})$-module $W_{(d-1,1)}$ with multiplicity $\ge d-1$. 
This multiplicity is in fact equal to $d-1$, because $d-1$ is the multiplicity of $W_{(d-1,1)}$ as a summand in 
$F_2$. Recall finally that for $\lambda=(\lambda_1,\lambda_2)\vdash m$, the multiplicity of $\chi^{\lambda}$ in the cocharacter sequence coincides with the multiplicity of $W_{\lambda}$ in 
$F_2/(I(\mathfrak{sl}_2(\mathbb{C}),\rho^{(d)})\cap F_2)$. 
\end{proof}


\end{document}